\documentclass[12pt]{amsart}
\usepackage{amsfonts, amssymb}
\usepackage{parskip, fullpage, verbatim}
\usepackage[colorlinks=true]{hyperref}
\usepackage{breakurl}

\newcommand\CA{{\mathcal A}}
\newcommand\CB{{\mathcal B}}

\newcommand\CIF{{\mathcal {IF}}} 
\newcommand\CIFM{{\mathcal {IFM}}}

\newcommand\BBC{{\mathbb C}}
\newcommand\BBK{{\mathbb K}}

\newcommand\BBZ{{\mathbb Z}}

\newcommand\codim{\operatorname{codim}}

\newcommand\Der{{\operatorname{Der}}}

\newcommand\pdeg{\operatorname{pdeg}}

\numberwithin{equation}{section}

\theoremstyle{plain}
\newtheorem{lemma}[equation]{Lemma}
\newtheorem{theorem}[equation]{Theorem}

\newtheorem{corollary}[equation]{Corollary}
\newtheorem{proposition}[equation]{Proposition}
\theoremstyle{definition}
\newtheorem{defn}[equation]{Definition}
\newtheorem{remark}[equation]{Remark}


\begin{document}

\title[Inductively free Multiderivations of Braid arrangements]
{Inductively free Multiderivations \\ of Braid arrangements}

\author[H. Conrad]{Henning Conrad}
\author[G. R\"ohrle]{Gerhard R\"ohrle}
\address
{Fakult\"at f\"ur Mathematik,
Ruhr-Universit\"at Bochum,
D-44780 Bochum, Germany}
\email{henning.conrad@rub.de}
\email{gerhard.roehrle@rub.de}


\allowdisplaybreaks

\begin{abstract}
The reflection arrangement of a Coxeter group 
is a well known instance of a free hyperplane arrangement.
In 2002, Terao showed that
equipped with a constant multiplicity 
each such reflection arrangement
gives rise to a free multiarrangement.
In this note we show that this
multiarrangment  
satisfies the stronger property of 
inductive freeness in case the 
Coxeter group is of type $A$.
\end{abstract}

\maketitle


\section{Introduction}

Arnold and independently Saito
proved that the 
reflection arrangements of Coxeter groups
are free, \cite[\S 6]{orlikterao:arrangements}. They
play a special role in the class of free hyperplane arrangements.
Terao's fundamental Addition-Deletion Theorem \cite{terao:freeI}
is a key tool in this theory. This leads to the 
important class of inductively free arrangements, 
\cite[\S 4]{orlikterao:arrangements}.
The significance of the latter lies in the fact that
within this class Terao's famous conjecture that 
freeness of an arrangement is combinatorial holds.

In his seminal work \cite{ziegler:multiarrangements}, Ziegler 
introduced the notion of multiarrangements and initiated the 
study of their freeness.  
In general, for a free hyperplane arrangement, 
an arbitrary multiplicity need not afford a 
free multiarrangement, 
e.g.\ see \cite[Ex.\ 14]{ziegler:multiarrangements}.

By constructing an explicit basis of the module of derivations, 
Terao showed in \cite{terao:multi} that
each Coxeter  arrangement  gives rise to a free multiarrangement
when endowed with a constant multiplicity.

In their ground breaking work 
\cite[Thm.\ 0.8]{abeteraowakefield:euler}, Abe, Terao and Wakefield
proved the 
Addition-Deletion Theorem for multiarrangements.
In analogy to the situation for simple arrangements
mentioned above, 
this naturally leads to the class of 
inductively free multiarrangements, 
see Definition \ref{def:indfree} below.
Since the computation of Euler multiplicities
is rather complicated, much less is known about 
this class of multiarrangments compared
to its cousin for simple arrangements as
indicated above.

Let $\CB_\ell$ be the braid arrangement in $\BBC^\ell$. It is the  
direct product of 
the empty 1-arrangement  $\Phi_1$
and the irreducible Coxeter arrangement $\CA_{\ell-1}$
of type $A_{\ell-1}$, \cite[\S 6.4]{orlikterao:arrangements}.
It follows from Definition \ref{def:indfree} and 
Theorem \ref{thm:products} that a multiplicity on 
$\CB_\ell$ is inductively free if and only if 
the corresponding multiplicity 
on the factor $\CA_{\ell-1}$  is inductively free.

Our main result shows that the irreducible Coxeter 
arrangement $\CA_{\ell-1}$ of type $A_{\ell-1}$ when equipped with a 
constant multiplicity 
is an inductively free multiarrangement.

\begin{theorem}
\label{thm:multibraid}
Let $\CB_\ell$ be the braid arrangement.
Then, for $m \in \BBZ_{\ge 1}$, 
the multiarrangement $(\CB_\ell, m)$ with defining polynomial
\[
Q(\CB_\ell, m) = \prod\limits_{1\le i<j\leq \ell}(x_i-x_j)^m
\]
is inductively free.
In particular, for the irreducible Coxeter arrangement 
$\CA_{\ell-1}$  of type $A_{\ell-1}$, the multiarrangement
$(\CA_{\ell-1}, m)$ is inductively free with exponents given by
\[
\exp (\CA_{\ell-1}, m) = 
\left\{\tfrac{m\ell}{2},\ldots,\tfrac{m\ell}{2}\right\} 
\]
for $m$ even, respectively 
\[
\exp (\CA_{\ell-1}, m) = 
\left\{\tfrac{(m-1)\ell}{2}+1,\tfrac{(m-1)\ell}{2}+2,
\ldots,\tfrac{(m-1)\ell}{2}+\ell-1\right\}
\]
for $m$ odd. 
\end{theorem}

The exponents in 
Theorem \ref{thm:multibraid}
have been determined by Terao
\cite[Thm.\ 1.1]{terao:multi}.
Note that $\ell$ is the Coxeter number of 
the irreducible Coxeter group of type  $A_{\ell-1}$
and $\{1, 2, \ldots, \ell-1\}$ is its set
of exponents, see \cite[V 6.2]{bourbaki:groupes}.

As a consequence of our proof of 
Theorem \ref{thm:multibraid}, we also obtain that certain 
non-constant multiplicities 
give rise to inductively free multiarrangements of the 
braid arrangement. These occur as restrictions in our induction tables.

\begin{corollary}
\label{cor:mixedmulti}
Let $\CA$ be the irreducible Coxeter 
arrangement of type $A_{\ell-1}$.
Then, for $m,q \in \BBZ_{\ge 1}$, 
the multiarrangement $(\CA; m, q)$ with defining polynomial
\[
Q(\CA;m,q)=
\prod\limits_{2 \le j\leq \ell}(x_1-x_j)^{m+q}\prod\limits_{2 \le i<j\leq \ell}(x_i-x_j)^m
\]
is inductively free with
\[
\exp(\CA;m,q)=\left\{\tfrac{m\ell}{2}+q,\ldots,\tfrac{m\ell}{2}+q\right\}
\]
when $m$ is even or
\[
\exp(\CA;m,q)=\left\{\tfrac{(m-1)\ell}{2}+1+q,\tfrac{(m-1)\ell}{2}+2+q,
\ldots,\tfrac{(m-1)\ell}{2}+\ell-1+q\right\}
\]
when $m$ is odd. 
\end{corollary}

It can be rather challenging  
to prove or disprove 
that a given 
arrangement is inductively free, e.g.\ see
\cite[Lem.\ 4.2]{amendhogeroehrle:indfree},
\cite[\S 5.2]{cuntz:indfree}, and  
\cite[Lem.\ 3.5]{hogeroehrle:indfree}.  
In principle, one might have to search through
all possible chains of free subarrangements.
We prove Theorem \ref{thm:multibraid} by 
exhibiting an induction table
of inductively free submultiarrangements,
see Remark \ref{rem:indtable}.

If $\CA$ is an inductively free simple arrangement, 
then for $m \in \BBZ_{\ge 1}$
the multiarrangement $(\CA, m)$
need not be inductively free in general
(indeed it need not even be free),
e.g.\ see \cite[Ex.\ 14]{ziegler:multiarrangements}.
So the situation for Coxeter groups as suggested by
Theorem \ref{thm:multibraid} is very special.

In \cite[Thm.\ 0.3]{abenudanumata}, 
Abe, Nuda and Numata determine a large class of 
free multiplicities of the braid arrangement.
The approach in \cite{abenudanumata} is by means of 
graphic arrangements.
In the notation of \emph{loc.~cit.}, if 
$m$ is even and $G$ is the empty graph 
and the parameters there are chosen to be
$k = m/2$, $n_1 = q$ and $n_2 = \ldots = n_\ell = 0$, 
repectively 
$m$ is odd and $G$ is the 
complete graph on $\ell$ vertices with only positive signs
and the parameters are chosen to be
$k = (m-1)/2$, $n_1 = q$ and $n_2 = \ldots = n_\ell = 0$, 
one recovers the 
freeness statements for the given 
multiplicities
in Corollary \ref{cor:mixedmulti}.
Moreover, since the proof of \cite[Thm.\ 0.3]{abenudanumata}
is based on the Addition-Deletion Theorem \ref{thm:add-del}
for multiarrangements, all the free multiarrangements
considered there are also inductively free, although 
this is not explicitly stated in \emph{loc.~cit}.

The graph theoretic approach employed 
in \cite{abenudanumata} is restricted to the braid arrangement.
In contrast, our treatment applies in the general setting 
of multiarrangements as such 
and in principle allows to be generalized to other Coxeter groups,
albeit general reflection arrangements are considerably 
more complicated than those of type $A$.
Therefore, it is likely that the 
reflection arrangement of any Coxeter group with 
constant multiplicity is also inductively free.
It is also natural to investigate inductive freeness for 
non-constant free multiplicities
of Coxeter arrangements, 
cf.~\cite{abeyoshinaga:quasi}, 
\cite{abeteraowakamiko:equivariant}.

We refer to \cite[Rem.\ 1.6]{terao:multi}
and \cite{yoshinaga:free04}
for the connection 
of the question of freeness of a
Coxeter arrangement $\CA$ endowed with a constant multiplicity 
and the question of freeness of extended Shi and 
extended Catalan arrangements. The latter was a 
conjecture of Edelman and Reiner and is
proved by Yoshinaga in \cite{yoshinaga:free04}.

\section{Recollection and Preliminaries}

\subsection{Hyperplane Arrangements}
\label{ssect:hyper}
Let $V = \BBK^\ell$ 
be an $\ell$-dimensional $\BBK$-vector space.
A \emph{hyperplane arrangement} is a pair
$(\CA, V)$, where $\CA$ is a finite collection of hyperplanes in $V$.
Usually, we simply write $\CA$ in place of $(\CA, V)$.
We write $|\CA|$ for the number of hyperplanes in $\CA$.
The empty arrangement in $V$ is denoted by $\Phi_\ell$.

The \emph{lattice} $L(\CA)$ of $\CA$ is the set of subspaces of $V$ of
the form $H_1\cap \dotsm \cap H_i$ where $\{ H_1, \ldots, H_i\}$ is a subset
of $\CA$. 
For $X \in L(\CA)$, we have two associated arrangements, 
firstly
$\CA_X :=\{H \in \CA \mid X \subseteq H\} \subseteq \CA$,
the \emph{localization of $\CA$ at $X$}, 
and secondly, 
the \emph{restriction of $\CA$ to $X$}, $(\CA^X,X)$, where 
$\CA^X := \{ X \cap H \mid H \in \CA \setminus \CA_X\}$.
Note that $V$ belongs to $L(\CA)$
as the intersection of the empty 
collection of hyperplanes and $\CA^V = \CA$. 

If $0 \in H$ for each $H$ in $\CA$, then 
$\CA$ is called \emph{central}.
If $\CA$ is central, then  
$T_\CA := \cap_{H \in \CA} H$ 
is the \emph{center} of $\CA$.
We have a \emph{rank} function on $L(\CA)$: $r(X) := \codim_V(X)$.
The \emph{rank} $r := r(\CA)$ of $\CA$ 
is the rank of $T_\CA$.

\subsection{Free Hyperplane Arrangements}
Let $S = S(V^*)$ be the symmetric algebra of the dual space $V^*$ of $V$.
If $x_1, \ldots , x_\ell$ is a basis of $V^*$, then we identify $S$ with 
the polynomial ring $\BBK[x_1, \ldots , x_\ell]$.
Letting $S_p$ denote the $\BBK$-subspace of $S$
consisting of the homogeneous polynomials of degree $p$ (along with $0$),
$S$ is naturally $\BBZ$-graded: $S = \oplus_{p \in \BBZ}S_p$, where
$S_p = 0$ in case $p < 0$.

Let $\Der(S)$ be the $S$-module of algebraic $\BBK$-derivations of $S$.
Using the $\BBZ$-grading on $S$, $\Der(S)$ becomes a graded $S$-module.
For $i = 1, \ldots, \ell$, 
let $D_i := \partial/\partial x_i$ be the usual derivation of $S$.
Then $D_1, \ldots, D_\ell$ is an $S$-basis of $\Der(S)$.
We say that $\theta \in \Der(S)$ is 
\emph{homogeneous of polynomial degree p}
provided 
$\theta = \sum_{i=1}^\ell f_i D_i$, 
where $f_i$ is either $0$ or homogeneous of degree $p$
for each $1 \le i \le \ell$.
In this case we write $\pdeg \theta = p$.

Let $\CA$ be an arrangement in $V$. 
Then for $H \in \CA$ we fix $\alpha_H \in V^*$ with
$H = \ker(\alpha_H)$.
The \emph{defining polynomial} $Q(\CA)$ of $\CA$ is given by 
$Q(\CA) := \prod_{H \in \CA} \alpha_H \in S$.

The \emph{module of $\CA$-derivations} of $\CA$ is 
defined by 
\[
D(\CA) := \{\theta \in \Der(S) \mid \theta(\alpha_H) \in \alpha_H S
\text{ for each } H \in \CA \} .
\]
We say that $\CA$ is \emph{free} if the module of $\CA$-derivations
$D(\CA)$ is a free $S$-module.

With the $\BBZ$-grading of $\Der(S)$, 
also $D(\CA)$ 
becomes a graded $S$-module,
\cite[Prop.\ 4.10]{orlikterao:arrangements}.
If $\CA$ is a free arrangement, then the $S$-module 
$D(\CA)$ admits a basis of $\ell$ homogeneous derivations, 
say $\theta_1, \ldots, \theta_\ell$, \cite[Prop.\ 4.18]{orlikterao:arrangements}.
While the $\theta_i$'s are not unique, their polynomial 
degrees $\pdeg \theta_i$ 
are unique (up to ordering). This multiset is the set of 
\emph{exponents} of the free arrangement $\CA$
and is denoted by $\exp \CA$.

\subsection{Multiarrangements}
\label{ssec:multi}
A \emph{multiarrangement}  is a pair
$(\CA, \nu)$ consisting of a hyperplane arrangement $\CA$ and a 
\emph{multiplicity} function
$\nu : \CA \to \BBZ_{\ge 0}$ associating 
to each hyperplane $H$ in $\CA$ a non-negative integer $\nu(H)$.
Alternately, the multiarrangement $(\CA, \nu)$ can also be thought of as
the multiset of hyperplanes
\[
(\CA, \nu) = \{H^{\nu(H)} \mid H \in \CA\}.
\]
We say that $\nu$ is a \emph{constant} multiplicity
provided there is some fixed $m \in \BBZ_{\ge 0}$ so that
$\nu(H) = m$ for every $H \in \CA$.
In that case we also 
say that $\nu$ is constant of \emph{weight} $m$ and 
frequently write 
$(\CA, m)$ in place of $(\CA, \nu)$.

The \emph{order} of the multiarrangement $(\CA, \nu)$ 
is the cardinality 
of the multiset $(\CA, \nu)$; we write 
$|\nu| := |(\CA, \nu)| = \sum_{H \in \CA} \nu(H)$.
For a multiarrangement $(\CA, \nu)$, the underlying 
arrangement $\CA$ is sometimes called the associated 
\emph{simple} arrangement, and so $(\CA, \nu)$ itself is  
simple if and only if $\nu(H) = 1$ for each $H \in \CA$. 

Let $\CA = \{H_1 , H_2, \ldots \}$
be a simple arrangement.
Then sometimes it 
is convenient to denote a multiplicity function $\nu$
on $\CA$ simply by the ordered tuple of its values 
$[\nu(H_1), \nu(H_2), \ldots ]$.

\begin{defn}
\label{def:submulti}
Let $\nu_i$ be a multiplicity of $\CA_i$ for $ i = 1,2$.
When viewed as multisets, suppose that 
$(\CA_1, \nu_1)$ is a subset of $(\CA_2, \nu_2)$.
Then we say that $(\CA_1, \nu_1)$ is a 
\emph{submultiarrangement} of $(\CA_2, \nu_2)$ and write 
$(\CA_1, \nu_1) \subseteq (\CA_2, \nu_2)$,
i.e.\ we have $\nu_1(H) \le \nu_2(H)$ for each $H \in \CA_1$.
\end{defn}

\begin{defn}
\label{def:localization}
Let $(\CA, \nu)$ be a multiarrangement in $V$ and let 
$X$ be in the lattice of $\CA$. The 
\emph{localization of $(\CA, \nu)$ at $X$} is $(\CA_X, \nu_X)$,
where $\nu_X = \nu |_{\CA_X}$.
\end{defn}

\subsection{Freeness of  multiarrangements}

Following Ziegler \cite{ziegler:multiarrangements},
we extend the notion of freeness to multiarrangements as follows.
The \emph{defining polynomial} $Q(\CA, \nu)$ 
of the multiarrangement $(\CA, \nu)$ is given by 
\[
Q(\CA, \nu) := \prod_{H \in \CA} \alpha_H^{\nu(H)},
\] 
a polynomial of degree $|\nu|$ in $S$.

The \emph{module of $\CA$-derivations} of $(\CA, \nu)$ is 
defined by 
\[
D(\CA, \nu) := \{\theta \in \Der(S) \mid \theta(\alpha_H) \in \alpha_H^{\nu(H)} S 
\text{ for each } H \in \CA\}.
\]
We say that $(\CA, \nu)$ is \emph{free} if 
$D(\CA, \nu)$ is a free $S$-module, 
\cite[Def.\ 6]{ziegler:multiarrangements}.

As in the case of simple arrangements,
$D(\CA, \nu)$ is a $\BBZ$-graded $S$-module and 
thus, if $(\CA, \nu)$ is free, there is a 
homogeneous basis $\theta_1, \ldots, \theta_\ell$ of $D(\CA, \nu)$.
The multiset of the unique polynomial degrees $\pdeg \theta_i$ 
forms the set of \emph{exponents} of the free multiarrangement $(\CA, \nu)$
and is denoted by $\exp (\CA, \nu)$.
It follows from Ziegler's analogue of Saito's criterion 
\cite[Thm.\ 8]{ziegler:multiarrangements} that 
$\sum \pdeg \theta_i = \deg Q(\CA, \nu) = |\nu|$. 

\begin{remark}
\label{rem:freeproducts}
A product of multiarrangements is free if and only if each factor is free:
using \cite[Lem.\ 1.3]{abeteraowakefield:euler}, 
the proof of \cite[Thm.\ 4.28]{orlikterao:arrangements}
readily extends to multiarrangements, thanks to 
Ziegler's analogue of Saito's criterion 
\cite[Thm.\ 8]{ziegler:multiarrangements}.
Moreover, in that case the set of exponents 
of the product is the union of the sets of exponents of the factors.
\end{remark}

\subsection{The Addition-Deletion Theorem for Multiarrangements}

We recall the construction from \cite{abeteraowakefield:euler}.

\begin{defn}
\label{def:Euler}
Let $(\CA, \nu) \ne \Phi_\ell$ be a multiarrangement. Fix $H_0$ in $\CA$.
We define the \emph{deletion}  $(\CA', \nu')$ and \emph{restriction} $(\CA'', \nu^*)$
of $(\CA, \nu)$ with respect to $H_0$ as follows.
If $\nu(H_0) = 1$, then set $\CA' = \CA \setminus \{H_0\}$
and define $\nu'(H) = \nu(H)$ for all $H \in \CA'$.
If $\nu(H_0) > 1$, then set $\CA' = \CA$
and define $\nu'(H_0) = \nu(H_0)-1$ and
$\nu'(H) = \nu(H)$ for all $H \ne H_0$.

Let $\CA'' = \{ H \cap H_0 \mid H \in \CA \setminus \{H_0\}\ \}$.
The \emph{Euler multiplicity} $\nu^*$ of $\CA''$ is defined as follows.
Let $Y \in \CA''$. Since the localization $\CA_Y$ is of rank $2$, the
multiarrangement $(\CA_Y, \nu_Y)$ is free, 
\cite[Cor.\ 7]{ziegler:multiarrangements}. 
According to 
\cite[Prop.\ 2.1]{abeteraowakefield:euler},
the module of derivations 
$D(\CA_Y, \nu_Y)$ admits a particular homogeneous basis
$\{\theta_Y, \psi_Y, D_3, \ldots, D_\ell\}$,
where $\theta_Y$ is identified by the 
property that $\theta_Y \notin \alpha_0 \Der(S)$
and $\psi_Y$ by the property that $\psi_Y\in \alpha_0 \Der(S)$,
where $H_0 = \ker \alpha_0$.
Then the Euler multiplicity $\nu^*$ is defined
on $Y$ as $\nu^*(Y) = \pdeg \theta_Y$.

We refer to $(\CA, \nu), (\CA', \nu')$ and $(\CA'', \nu^*)$ 
as the \emph{triple} of $(\CA, \nu)$ with respect to $H_0$. 
\end{defn}

\begin{theorem}
[{\cite[Thm.\ 0.8]{abeteraowakefield:euler}}
Addition-Deletion-Theorem for Multiarrangements]
\label{thm:add-del}
Suppose that $(\CA, \nu) \ne \Phi_\ell$.
Fix $H_0$ in $\CA$ and 
let  $(\CA, \nu), (\CA', \nu')$ and  $(\CA'', \nu^*)$ be the triple with respect to $H_0$. 
Then any  two of the following statements imply the third:
\begin{itemize}
\item[(i)] $(\CA, \nu)$ is free with $\exp (\CA, \nu) = \{ b_1, \ldots , b_{\ell -1}, b_\ell\}$;
\item[(ii)] $(\CA', \nu')$ is free with $\exp (\CA', \nu') = \{ b_1, \ldots , b_{\ell -1}, b_\ell-1\}$;
\item[(iii)] $(\CA'', \nu^*)$ is free with $\exp (\CA'', \nu^*) = \{ b_1, \ldots , b_{\ell -1}\}$.
\end{itemize}
\end{theorem}

The following is part of 
\cite[Prop.\ 4.1]{abeteraowakefield:euler}
relevant for our purposes.

\begin{proposition}
\label{prop:Euler}
Let $H_0\in\CA$,  $\CA'' = \CA^{H_0}$ and let  $X\in\CA''$.
Let $\nu$ be a multiplicity on $\CA$.
Let $\nu_0 = \nu(H_0)$. Further let $k=|\CA_X|$ and $\nu_1=\max\{\nu(H) \mid H\in\CA_X\setminus\{H_0\}\}$.
\begin{itemize}
\item[(1)] If $k=3,\ 2\nu_0\le |\nu_X|$, and $2\nu_1\le |\nu_X|$, then 
$\nu^*(X)=\left\lfloor{|\nu_X|}/{2}\right\rfloor$.
\item[(2)] If $k=2$, then $\nu^*(X) = \nu_1$.
\item[(3)] If $2\nu_1\ge |\nu_X|-1$, then $\nu^*(X) =\nu_1$.
\end{itemize}
\end{proposition}

\subsection{Inductive Freeness for Multiarrangements}
\label{subsec:indutive}
As in the simple case, Theorem \ref{thm:add-del} motivates 
the notion of inductive freeness. 

\begin{defn}[{\cite[Def.\ 0.9]{abeteraowakefield:euler}}]
\label{def:indfree}
The class $\CIFM$ of \emph{inductively free} multiarrangements 
is the smallest class of arrangements subject to
\begin{itemize}
\item[(i)] $\Phi_\ell \in \CIFM$ for each $\ell \ge 0$;
\item[(ii)] for a multiarrangement $(\CA, \nu)$, if there exists a hyperplane $H_0 \in \CA$ such that both
$(\CA', \nu')$ and $(\CA'', \nu^*)$ belong to $\CIFM$, and $\exp (\CA'', \nu^*) \subseteq \exp (\CA', \nu')$, 
then $(\CA, \nu)$ also belongs to $\CIFM$.
\end{itemize}
\end{defn}

\begin{remark}[{\cite[Rem.\ 0.10]{abeteraowakefield:euler}}]
\label{rem:indfree}
The intersection of $\CIFM$ with the class of 
simple arrangements is the class $\CIF$ of inductively free arrangements.
\end{remark}

\begin{remark}
\label{rem:rank2}
As in the simple case, if $r(\CA) \le 2$,
then $(\CA, \nu)$  is inductively free,  
\cite[Cor.~7]{ziegler:multiarrangements}.
\end{remark}

\begin{remark}
\label{rem:indtable}
In analogy to the simple case, 
cf.~\cite[\S 4.3, p.~119]{orlikterao:arrangements}, 
\cite[Rem.\ 2.9]{hogeroehrle:indfree}, 
it is possible to describe an 
inductively free multiarrangement $(\CA, \nu)$
by means of a so called 
\emph{induction table}. 
In this process we start with an inductively free multiarrangement
(frequently $\Phi_\ell$) and add hyperplanes successively ensuring that 
part (ii) of Definition \ref{def:indfree} is satisfied.
We refer to this process as \emph{induction of hyperplanes}.
This procedure amounts to 
choosing a total order on the multiset $(\CA, \nu)$, say 
$\CA = \{H_1, \ldots, H_n\}$, where $n = |\nu|$, 
so that each of the submultiarrangements 
$\CA_0 := \Phi_\ell$, 
$(\CA_i, \nu_i) := \{H_1, \ldots, H_i\}$
(viewed again as multiset)
and each of the restrictions $(\CA_i^{H_i}, \nu_i^*)$ 
is inductively free for $i = 1, \ldots, n$.
As in the simple case, 
in the associated induction table 
we record in the $i$-th row the information 
of the $i$-th step of this process, by 
listing $\exp (\CA_i',\nu_i') = \exp (\CA_{i-1},\nu_{i-1})$, 
the defining form $\alpha_{H_i}$ of $H_i$, 
as well as $\exp (\CA_i'', \nu_i^*) = \exp (\CA_i^{H_i}, \nu_i^*)$, 
for $i = 1, \ldots, n$.
Frequently, we refer to a 
triple $(\CA_i,\nu_i)$,  $(\CA_{i-1},\nu_{i-1})$, and  $(\CA_i^{H_i}, \nu_i^*)$ 
in such an induction table  
as an \emph{inductive triple}.
In addition we also record the Euler multiplicity and 
in part the relevant data from Proposition \ref{prop:Euler}.  
For instance, see Tables 
\ref{indtable1} up to \ref{indtable7} below. 
\end{remark}

We also require the 
following result from 
\cite[Thm.\ 1.4]{hogeroehrleschauenburg:free};
this extends the compatibility of freeness with products
from Remark \ref{rem:freeproducts} to inductive freeness.

\begin{theorem}
\label{thm:products}
A product of multiarrangements belongs to 
$\CIFM$ if and only if 
each factor belongs to $\CIFM$.
\end{theorem}

\begin{remark}
Since localization is compatible with the product 
construction, it follows 
from the definition of the Euler multiplicity that  
it is also compatible with this product construction.
In particular, the Euler multiplicity of the 
restricion of a product to a hyperplane only depends
on the relevant factor.
We use this fact throughout without further comment.
\end{remark}

\section{Proof of Theorem \ref{thm:multibraid}}
\label{sec:proof}

In order to prove Theorem \ref{thm:multibraid}, 
we perform an induction of hyperplanes, 
see Remark \ref{rem:indtable}.
By \cite[Prop.\ 6.73]{orlikterao:arrangements},
every restricted arrangement $\CA''$ is of Coxeter type $A$ again.
However, calculating the corresponding Euler multiplicities of these 
restrictions, we see that we do not always get a constant multiplicity. 
If $(\CA,\nu) = (\CA,m)$ has a constant multiplicity of weight $m$, then 
during the induction of hyperplanes, 
$(\CA'',\nu^*)$ has multiplicity 
given by the following 
defining polynomial
\[
Q(\CA_{\ell-2};m,q):=
\prod\limits_{1<j\leq \ell-1}(x_1-x_j)^{m+q}\prod\limits_{2 \le i<j \leq \ell-1}(x_i-x_j)^m
\]
for some non-negative integer $q$. 

If $q=0$ (i.e.\ when $\nu^*$ is a constant multiplicity), 
then the exponents are given by Theorem \ref{thm:multibraid}.
In any case, irrespective of being able to determine 
the exponents in our induction, we do not know a priori 
whether or not the restricted multiarrangements
that occur are inductively free. 
In this context, the next result is very useful. 
It states that such arrangements with described 
multiplicities are indeed inductively free assuming 
Theorem \ref{thm:multibraid} holds.

\begin{lemma}
\label{lem:mixed}
Let $\CA$ be the Coxeter 
arrangement of type $A_{\ell-1}$
and let $\nu \colon\CA \to \mathbb Z_{\geq 0}$ be a constant multiplicity
of weight $m$.
Suppose that the multiarrangement $(\CA,\nu)$ is inductively free. 
Then, for any $q\in\mathbb Z_{\geq 0}$, 
the multiarrangement $(\CA;m,q)$ with defining polynomial
\[
Q(\CA;m,q) :=\prod\limits_{2 \le j\leq \ell}(x_1-x_j)^{m+q}
\prod\limits_{2 \le i<j \leq \ell}(x_i-x_j)^m
\]
is inductively free with
\[
\exp(\CA;m,q)=\left\{\tfrac{m \ell}{2}+q,\ldots,\tfrac{m \ell}{2}+q\right\}
\]
when $m$ is even, respectively 
\[
\exp(\CA;m,q)=\left\{\tfrac{(m-1)\ell}{2}+1+q, \tfrac{(m-1)\ell}{2}+2+q,
\ldots,\tfrac{(m-1)\ell}{2}+\ell-1+q\right\}
\]
when $m$ is odd.
\end{lemma}

\begin{proof} 
Let $\CA_{\ell-1}$ be
the Coxeter arrangement of type $A_{\ell-1}$. 
We argue by induction on $\ell$. 
For $\ell=2$, it follows from 
Remark \ref{rem:rank2}
that $(\CA_1;m,q)$
is inductively free. We have
\[
Q(\CA_1;m,q)=(x_1 - x_2)^{m+q}.
\]
This is a Coxeter arrangement of type $A_1$ with 
a constant multiplicity $m+q$ and so, 
its set of exponents is $\{m+q\}$, as given by 
Theorem \ref{thm:multibraid}, thanks to \cite[Thm.\ 1.1]{terao:multi}.
Note that this does not 
depend on the parity of $m+q$.
So the result follows for $\ell = 2$.
 
Strictly speaking, the case $\ell = 3$ is not 
necessary in our induction. It is however very instructive 
to see the arguments in this case, as this is an  
instance of a non-constant multiplicity.

For $\ell = 3$ we see again by Remark \ref{rem:rank2} that
$(\CA_2;m,q)$ is inductively free. 
But this time we do not have a constant multiplicity. 
Here the multiplicity is given by $[m+q,m+q,m]$. 
Therefore, we perform an induction of hyperplanes, 
starting with the case in which $m$ is even (including $m=0$). 
By assumption of the lemma, the multiarrangement 
$(\CA_2,\nu) = (\CA_2,m)$ is 
inductively free. 
Therefore, we may initialize the induction table with the 
multiarrangement 
$(\CA_2,m)$ with
\[
\exp(\CA_2,m)=\left\{\tfrac{3m}{2},\tfrac{3m}{2}\right\},
\]
where the exponents are again given by \cite[Thm.\ 1.1]{terao:multi}.
Our aim is to add the hyperplanes of type 
$\ker(x_1-x_j)$ ($j=2,3$) $q$ times successively. 
In this $3$-dimensional case, determining the 
restriction in each step is very simple. 
We have $\CA''=\{x_1=x_2=x_3\}$ (except for 
the first step of the case $m=0$ where 
$\CA''=\Phi_2$). 
The Euler multiplicity can easily be calculated using 
Proposition \ref{prop:Euler}(1), 
because we have $k=3$ in every step. 
The resulting multiarrangement $(\CA'',\nu^*)$ 
is always inductively free because it has rank $1$. 
Since $\CA''\cong A_1$   
is again a Coxeter arrangement of type $A_1$ with 
a constant multiplicity $|\nu^*|$, 
its exponents are as given by 
Theorem \ref{thm:multibraid}.
In the first step we obtain 
$\nu^*(X)=\left\lfloor\frac{3m+1}{2}\right\rfloor$,
so the multiplicity of the single hyperplane 
is ${3m}/{2}$ and so
$\exp (\CA'',\nu^*) = \{{3m}/{2}\}$.
Applying Theorem \ref{thm:add-del},
we can easily determine the exponents of the 
new multiarrangement in every step, see Table \ref{indtable1}.

\begin{table}[ht!b]\small
\renewcommand{\arraystretch}{1.5}
\begin{tabular}{llll}
  \hline
$\exp (\CA',\nu')$ & $\alpha_H$ & $\exp (\CA'', \nu^*)$ & Euler multiplicities\\
  \hline
  \hline
  $\frac{3m}{2}, \frac{3m}{2}$ & $x_1-x_2$ & $\frac{3m}{2}$ & $\nu_0=m+1, \nu_1=m$,\\
         &  &  & $|\nu_X|=3m+1, \nu^*(X)=\left\lfloor\frac{3m+1}{2}\right\rfloor$ \\
  \hline
  $\frac{3m}{2},\frac{3m}{2}+1$ & $x_1-x_3$ & $\frac{3m}{2}+1$ & $\nu_0=\nu_1=m+1$,\\
      &  & & $|\nu_X|=3m+2, \nu^*(X)=\left\lfloor\frac{3m+2}{2}\right\rfloor$   \\
  \hline
$\frac{3m}{2}+1,\frac{3m}{2}+1$ & $x_1-x_2$ & $\frac{3m}{2}+1$ & $\nu_0=m+2, \nu_1=m+1$,\\
  &  &  & $|\nu_X|=3m+3, \nu^*(X)=\left\lfloor\frac{3m+3}{2}\right\rfloor$   \\
  \hline
  $\frac{3m}{2}+1,\frac{3m}{2}+2$ & $x_1-x_3$  & $\frac{3m}{2}+2$ & $\nu_0=\nu_1=m+2$,                                              \\
  & & & $|\nu_X|=3m+4, \nu^*(X)=\left\lfloor\frac{3m+4}{2}\right\rfloor$   \\
  \hline
  $\vdots$      & $\vdots$   & $\vdots$                 & $\vdots$  \\
  \hline
  $\frac{3m}{2}+q-1,\frac{3m}{2}+q-1$ & $x_1-x_2$  & $\frac{3m}{2}+q-1$       & $\nu_0=m+q, \nu_1=m+q-1$,  \\
   &  &    & $|\nu_X|=3m+2q-1$,  $\nu^*(X)=\left\lfloor\frac{3m+2q-1}{2}\right\rfloor$            \\ 
  \hline
  $\frac{3m}{2}+q-1,\frac{3m}{2}+q$   & $x_1-x_3$  & $\frac{3m}{2}+q$         & $\nu_0=\nu_1=m+q$,                                              \\
      &  &  & $|\nu_X|=3m+2q, \nu^*(X)=\left\lfloor\frac{3m+2q}{2}\right\rfloor$ \\
  \hline
  $\frac{3m}{2}+q,\frac{3m}{2}+q$                 \\                            
\hline
\end{tabular}
\medskip
\caption{Lemma \ref{lem:mixed}: 
Induction of hyperplanes for  $\ell=3$ and $m$ even}
\label{indtable1} 
\end{table}

The case where $m$ is odd is treated in a similar way. Starting with
\[
\exp(\CA_2,m)=\left\{\tfrac{3(m-1)}{2}+1,\tfrac{3(m-1)}{2}+2\right\}
\]
we add the hyperplanes of 
type $\ker(x_1-x_j)$ ($j=2,3$) $q$ times successively until we get
\[
\exp(\CA_2;m,q)=\left\{\tfrac{3(m-1)}{2}+q+1,\tfrac{3(m-1)}{2}+q+2\right\}.
\]

Now suppose that $m$ is even, $\ell>3$, and that the 
statement of the lemma holds for all values of $q$ for smaller ranks.
By hypothesis of the lemma, the 
multiarrangement $(\CA_{\ell-1},m)$ 
is inductively free with exponents  
\[
\exp(\CA_{\ell-1},m)=\left\{\tfrac{m \ell}{2},\ldots,\tfrac{m \ell}{2}\right\}.
\]
We initialize our induction table with this inductively free 
multiarrangement.
Then 
we consider the restriction to $\ker(x_1-x_j)$ without loss, 
so that  
\[
\CA''=\{x_1=x_j=x_a;x_1=x_j,x_b=x_c\}\cong\CA_{\ell-2},
\]
where $1 < j\leq \ell$ and $a,b,c\neq 1,j$ and $b \ne c$. 
We denote the members of $\CA''$ by
$Y^{j}_a:=\{x_1=x_j=x_a\}$ and 
$Y^{j}_{b,c}:=\{x_1=x_j,x_b=x_c\}$.
In the first step we restrict to $\ker(x_1-x_2)$ 
and calculate the Euler multiplicities 
using Proposition \ref{prop:Euler}(1) and (2) as follows: 
for $Y^2_a$ we have
\[
Q(\CA_{Y^2_a},\nu_{Y^2_a})=(x_1-x_2)^{m+1}(x_1-x_a)^m(x_2-x_a)^m.
\]
Hence, by Proposition \ref{prop:Euler}(1),
\[
\nu^*(Y^2_a)=\left\lfloor\tfrac{3m+1}{2}\right\rfloor=\tfrac{3m}{2}.
\]
For $Y^2_{b,c}$ we have
\[
Q(\CA_{Y^2_{b,c}},\nu_{Y^2_{b,c}})=(x_1-x_2)^{m+1}(x_b-x_c)^m.
\]
Consequently, by Proposition \ref{prop:Euler}(2), we get 
$\nu^*(Y^2_{b,c})=m$. 
Therefore, the Euler multiplicity $\nu^*$ on $\CA''$ is given by 
\[
\left[\tfrac{3m}{2},\ldots,\tfrac{3m}{2},m,\ldots,m\right] = 
\left[m+\tfrac{m}{2},\ldots, m+\tfrac{m}{2},m,\ldots,m\right].
\]
Due to our induction hypothesis 
and the fact that $\CA''\cong\CA_{\ell-2}$, 
the resulting multiarrangement $(\CA'',\nu^*)$ is inductively free with
\[
\exp(\CA'',\nu^*) = 
\left\{\tfrac{m (\ell-1)}{2} + \tfrac{m}{2},\ldots,\tfrac{m (\ell-1)}{2} + \tfrac{m}{2}\right\} =
\left\{\tfrac{m \ell}{2},\ldots,\tfrac{m \ell}{2}\right\}
\]
(note here $q = m/2$).
By the addition part of Theorem \ref{thm:add-del}
and Definition \ref{def:indfree}, we see that
$(\CA',\nu')$ in the next step is inductively free with
\[
\exp(\CA',\nu')=
\left\{\tfrac{m \ell}{2},\ldots,\tfrac{m \ell}{2},\tfrac{m \ell}{2}+1\right\}.
\]
Restricting to $\ker(x_1-x_3)$ leads to 
a similar arrangement as in the previous step, 
but this time the multiplicity $\nu^*$ of $\CA''$ is given by 
\[
\left[\tfrac{3m}{2}+1,\tfrac{3m}{2},\ldots,\tfrac{3m}{2},m,\ldots,m\right].
\]

\begin{table}[ht!b]\small
\renewcommand{\arraystretch}{1.5}
\begin{tabular}{llll}
  \hline
$\exp (\CA',\nu')$ & $\alpha_H$ & $\exp (\CA'', \nu^*)$ and 
Euler multiplicity\\
  \hline
  \hline
  $\underbrace{\tfrac{m \ell}{2},\ldots,\tfrac{m \ell}{2}}_{(\ell-1)\text{ times}}$      
 & $x_1-x_2$ & $\underbrace{\tfrac{m \ell}{2},\ldots,\tfrac{m \ell}{2}}_{(\ell-2)\text{ times}}$ 
\ \ \ $\left[\tfrac{3m}{2},\ldots,\tfrac{3m}{2},m,\ldots,m\right]$ \\
  \hline
  $\frac{m \ell}{2},\ldots,\frac{m \ell}{2},\frac{m \ell}{2}+1$    
  & $x_1-x_3$ & $\frac{m \ell}{2},\ldots,\frac{m \ell}{2},\frac{m \ell}{2}+1$ \\         
     & & $\left[\frac{3m}{2}+1,\frac{3m}{2},\ldots,\frac{3m}{2},m,\ldots,m\right]$\\
    \hline
$\frac{m \ell}{2},\ldots,\frac{m \ell}{2},\frac{m \ell}{2}+1,\frac{m \ell}{2}+1$ 
& $x_1-x_4$ & $\frac{m \ell}{2},\ldots,\frac{m \ell}{2},\frac{m \ell}{2}+1,\frac{m \ell}{2}+1$ \\ 
 & & $\left[\frac{3m}{2}+1,\frac{3m}{2}+1,\frac{3m}{2},\ldots,\frac{3m}{2},m,\ldots,m\right]$ \\
    \hline
  $\vdots$ & $\vdots$ & $\vdots$   \\
      \hline
  $\frac{m \ell}{2},\frac{m \ell}{2}+1,\ldots,\frac{m \ell}{2}+1$    & $x_1-x_\ell$ & $\frac{m \ell}{2}+1,\ldots,\frac{m \ell}{2}+1$ \\
    & & $\left[\frac{3m}{2}+1,\ldots,\frac{3m}{2}+1,m,\ldots,m\right]$ \\
      \hline 
  $\frac{m \ell}{2}+1,\ldots,\frac{m \ell}{2}+1$  & $x_1-x_2$ & $\ldots$  \\
        \hline 
  $\vdots$ & $\vdots$ & $\vdots$                  \\
        \hline 
  $\frac{m \ell}{2}+q-1,\ldots,\frac{m \ell}{2}+q-1$        
     & $x_1-x_2$ &    $\frac{m \ell}{2}+q-1,\ldots,\frac{m \ell}{2}+q-1$ \\                         
      & & $\left[\frac{3m}{2}+q-1,\ldots,\frac{3m}{2}+q-1,m,\ldots,m\right]$ \\
         \hline 
  $\frac{m \ell}{2}+q-1,\ldots,\frac{m \ell}{2}+q-1, \frac{m \ell}{2}+q$      
      & $x_1-x_3$ &     $\frac{m \ell}{2}+q-1,\ldots,\frac{m \ell}{2}+q-1,\frac{m \ell}{2}+q$ \\
    &      &          $\left[\frac{3m}{2}+q,\frac{3m}{2}+q-1,\ldots,\frac{3m}{2}+q-1, m,\ldots,m\right]$ \\
          \hline 
  $\vdots$ & $\vdots$ & $\vdots$    \\    
          \hline 
  $\frac{m \ell}{2}+q-1,\frac{m \ell}{2}+q,\ldots, \frac{m \ell}{2}+q$   & $x_1-x_\ell$ &    
            $\frac{m \ell}{2}+q,\ldots,\frac{m \ell}{2}+q$ \\ 
  &  &   $\left[\frac{3m}{2}+q,\ldots,\frac{3m}{2}+q,m,\ldots,m\right]$ \\
          \hline 
  $\frac{m \ell}{2}+q,\ldots,\frac{m \ell}{2}+q$ & &   \\
\hline
\end{tabular}
\medskip
\caption{Lemma \ref{lem:mixed}: 
Induction of hyperplanes for  $\ell >3$ and $m$ even}
\label{indtable2} 
\end{table}

Obviously, such multiplicities occur 
while adding hyperplanes of the type 
$Y^3_a$ to the underlying multiarrangement 
with constant multiplicity $m$. In this case, 
the hyperplanes of type $Y^3_a$ have been 
added $\frac{m}{2}$ times except for 
$Y^3_2=\{x_1=x_2=x_3\}$ which already 
has multiplicity $\frac{m}{2}+1$ and we have
\[
\exp(\CA'',\nu^*)=\left\{\tfrac{m \ell}{2},\ldots,\tfrac{m \ell}{2},\tfrac{m \ell}{2}+1\right\}.
\]
Continuing in the same way, we can 
easily complete our induction of 
hyperlanes when $m$ is even, see Table \ref{indtable2}.

The case where $m$ is odd is again 
treated in an analogous way.
By hypothesis of the lemma, the 
multiarrangement $(\CA_{\ell-1},m)$ 
is inductively free with exponents  
\[
\exp(\CA_{\ell-1},m)=\left\{\tfrac{(m-1) \ell}{2}+1,\ldots,\tfrac{(m-1) \ell}{2}+\ell-1\right\}.
\]
We initialize our induction table with this inductively free 
multiarrangement.
The Euler multiplicities can then 
be calculated again using 
Proposition \ref{prop:Euler}(1) and (2),
see Table \ref{indtable3} for details.

\begin{table}[ht!b]\small
\renewcommand{\arraystretch}{1.5}
\begin{tabular}{llll}
  \hline
$\exp (\CA',\nu')$ & $\alpha_H$ & $\exp (\CA'', \nu^*)$ and 
Euler multiplicities\\
  \hline
  \hline
  $\frac{(m-1)\ell}{2}+1,\ldots,\frac{(m-1)\ell}{2}+\ell-1$  & $x_1-x_2$ & 
     $\frac{(m-1)\ell}{2}+2,\ldots,\frac{(m-1)\ell}{2}+\ell-1$ \\   
  &     & $\left[\frac{3m+1}{2},\ldots,\frac{3m+1}{2},m,\ldots,m\right]$ \\
  \hline
  $\frac{(m-1)\ell}{2}+2,\frac{(m-1)\ell}{2}+2$,      & $x_1-x_3$ &
     $\frac{(m-1)\ell}{2}+2,\ldots,\frac{(m-1)\ell}{2}+\ell-1$   \\ 
       $\frac{(m-1)\ell}{2}+3,\ldots,\frac{(m-1)\ell}{2}+\ell-1$         &     &
    $\left[\frac{3m+1}{2},\ldots,\frac{3m+1}{2},m,\ldots,m\right]$ \\
  \hline
  $\vdots$ & $\vdots$ & $\vdots$     \\
  \hline
  $\frac{(m-1)\ell}{2}+2,\ldots,\frac{(m-1)\ell}{2}+\ell-1, \frac{(m-1)\ell}{2}+\ell-1$ & $x_1-x_\ell$ & 
   $\frac{(m-1)\ell}{2}+2,\ldots,\frac{(m-1)\ell}{2}+\ell-1$ \\       
            &  &    $\left[\frac{3m+1}{2},\ldots,\frac{3m+1}{2},m,\ldots,m\right]$\\
  \hline
  $\frac{(m-1)\ell}{2}+2,\ldots,\frac{(m-1)\ell}{2}+\ell$  &  $x_1-x_2$  &  $\ldots$    \\
  \hline
  $\vdots$ & $\vdots$ & $\vdots$   \\
  \hline
  $\frac{(m-1)\ell}{2}+q,\ldots, \frac{(m-1)\ell}{2}+\ell-2+q$   & $x_1-x_2$ &  
   $\frac{(m-1)\ell}{2}+1+q,\ldots,\frac{(m-1)\ell}{2}+\ell-2+q$ \\        
       & & $\left[\frac{3m-1}{2}+q,\ldots,\frac{3m-1}{2}+q,m,\ldots,m\right]$ \\
  \hline
  $\frac{(m-1)\ell}{2}+1+q,\frac{(m-1)\ell}{2}+1+q,$   & $x_1-x_3$ &      
    $\frac{(m-1)\ell}{2}+1+q,\ldots,\frac{(m-1)\ell}{2}+\ell-2+q$ \\
  $\frac{(m-1)\ell}{2}+2+q,\ldots, \frac{(m-1)\ell}{2}+\ell-2+q$  & &
   $\left[\frac{3m-1}{2}+q,\ldots,\frac{3m-1}{2}+q,m,\ldots,m\right]$ \\
    \hline
  $\vdots$ & $\vdots$ & $\vdots$   \\    
    \hline
  $\frac{(m-1)\ell}{2}+1+q,\ldots,$   & $x_1-x_\ell$ & 
   $\frac{(m-1)\ell}{2}+1+q,\ldots,\frac{(m-1)\ell}{2}+\ell-2+q$ \\
  $\frac{(m-1)\ell}{2}+\ell-2+q, \frac{(m-1)\ell}{2}+\ell-2+q$    & &
  $\left[\frac{3m-1}{2}+q,\ldots,\frac{3m-1}{2}+q,m,\ldots,m\right]$ \\
    \hline
  $\frac{(m-1)\ell}{2}+1+q,\ldots, \frac{(m-1)\ell}{2}+\ell-1+q$   &   &  \\
\hline
\end{tabular}
\medskip
\caption{Lemma \ref{lem:mixed}: 
Induction of hyperplanes for  $\ell>3$ and $m$ odd}
\label{indtable3} 
\end{table}

This completes the proof of the lemma.
\end{proof}

\vfill
\eject

We prove Theorem \ref{thm:multibraid} 
by induction on the rank $\ell$. 
For $\ell=2$,  $\CA$ is a Coxeter arrangement of type $A_1$
and the multiarrangement 
$(\CA,m)$ is inductively free 
thanks to Remark \ref{rem:rank2},
with $ \exp (\CA,m) = \{m\}$, 
thanks to \cite[Thm.\ 1.1]{terao:multi}.

Now let $\ell = 3$. 
The underlying simple arrangement $\CA_2$ is inductively free 
due to Remark \ref{rem:rank2}
with $\exp(\CA_2)=\{1,2\}$. 
Thus we initialize our induction table 
with the simple inductively free arrangement $\CA_2 = (\CA_2, 1)$.
In our induction of hyperplanes, each of the three hyperplanes 
is added in turn until each has multiplicity $m$. 
Since in every step $\CA''\cong\CA_1$ 
is a Coxeter arrangement of type $A_1$ necessarily with 
a constant multiplicity, 
we readily obtain $\exp(\CA'',\nu^*)=\{|\nu^*|\}$, 
thanks to \cite[Thm.\ 1.1]{terao:multi}.
It is again very easy to determine 
the multiplicity $\nu^*$ at each step, using 
Proposition \ref{prop:Euler}(1); see
Table \ref{indtable4}.

\begin{table}[ht!b]\small
\renewcommand{\arraystretch}{1.5}
\begin{tabular}{llll}
  \hline
$\exp (\CA',\nu')$ & $\alpha_H$ & $\exp (\CA'', \nu^*)$ \\
  \hline
  \hline
  $1,2$                                      & $x_1-x_2$  & $2$                     \\
  $2,2$                                      & $x_1-x_3$  & $2$                     \\
  $2,3$                                      & $x_2-x_3$  & $3$                     \\
  $3,3 $                                     & $x_1-x_2$  & $3$                     \\
  $3,4$                                      & $x_1-x_3$  & $4$                     \\
  $4,4$                                      & $x_2-x_3$  & $4$                     \\ 
  $\vdots$                                   & $\vdots$   & $\vdots$                \\
  when $m$ is even:                &          &                       \\
  \hline
  $\frac{3m}{2}-2,\frac{3m}{2}-1$            & $x_1-x_2$  & $\frac{3m}{2}-1$        \\ 
  $\frac{3m}{2}-1,\frac{3m}{2}-1$            & $x_1-x_3$  & $\frac{3m}{2}-1$        \\
  $\frac{3m}{2}-1,\frac{3m}{2}$              & $x_2-x_3$  & $\frac{3m}{2}$          \\
  $\frac{3m}{2},\frac{3m}{2}$                &          &                       \\
  \hline
when $m$ is odd:                  &          &                       \\
  \hline
  $\frac{3(m-1)}{2},\frac{3(m-1)}{2}$        & $x_1-x_2$  & $\frac{3(m-1)}{2}$      \\
  $\frac{3(m-1)}{2},\frac{3(m-1)}{2}+1$      & $x_1-x_3$  & $\frac{3(m-1)}{2}+1$    \\
  $\frac{3(m-1)}{2}+1,\frac{3(m-1)}{2}+1$    & $x_2-x_3$  & $\frac{3(m-1)}{2}+1$    \\
  $\frac{3(m-1)}{2}+1,\frac{3(m-1)}{2}+2$    &          &                       \\
  \hline
\end{tabular}
\medskip
\caption{Theorem \ref{thm:multibraid}; induction of hyperplanes for $\ell=3$}
\label{indtable4} 
\end{table}

Now suppose that $\ell > 3$ and that the statement of the theorem holds for 
smaller ranks. In particular, the multiarrangement $(\CA_{\ell-2},\nu) = (\CA_{\ell-2},m)$ 
with constant multiplicity $m$ is inductively free.
By Theorem \ref{thm:products}, 
the multiarrangement $(\CA_{\ell-2},m)\times\Phi_1$ is inductively free as well. 
It has exponents $\{0,\exp(\CA_{\ell-2},m)\}$. 
In our induction of hyperplanes we now add the hyperplanes of 
type $\ker(x_i-x_\ell)$ (for $1\leq i<\ell$) $m$ times. 
The first $\frac{m}{2}$, respectively $\frac{m+1}{2}$ rounds adding 
those hyperplanes, the parity of $m$ does not matter.
In order to describe the restrictions we use the following notation.
We denote the members of $\CA''$ by
$Y^{j}_a:=\{x_j=x_\ell=x_a\}$ and 
$Y^{j}_{b,c}:=\{x_j=x_\ell,x_b=x_c\}$, 
where $1 \le j < \ell$ and $a,b,c\neq j, \ell$ and $b \ne c$.
Starting with $\ker(x_1-x_\ell)$ we have
\[
\CA''=\{x_1=x_\ell=x_a;x_1=x_\ell,x_b=x_c\}
\]
and of course $\CA''\cong\CA_{\ell-2}$. 
Using Proposition \ref{prop:Euler}(2), 
we get $\nu^*(Y^1_a)=\nu^*(Y^1_{b,c})=m$ 
which implies that $(\CA'',\nu^*)$ is a 
multiarrangement with a constant multiplicity $m$ and it is of Coxeter type 
$A_{\ell-2}$, hence it is inductively free due to our induction hypothesis. 
Continuing on, restricting to $\ker(x_i-x_\ell)$ for $2 \le i \le \ell-1$, 
we need to make use of Proposition \ref{prop:Euler}(2) and (3),
to derive that again $(\CA'',\nu^*)$ is a 
multiarrangement with a constant multiplicity $m$ and 
at each step $\CA''$ it is still of Coxeter type 
$A_{\ell-2}$, hence it is inductively free due to our induction hypothesis.
After the first round of adding hyperplanes the 
set of exponents of the new multiarrangement is 
$\{\ell-1, \exp(\CA_{\ell-2},m)\}$ and the multiplicity is
\[
[m,\ldots,m,\underbrace{1,\ldots,1}_{(\ell-1) \text{ times}}],
\]
because the hyperplanes $\ker(x_i-x_\ell)$ now have multiplicity $1$. 

\begin{table}[ht!b]\small
\renewcommand{\arraystretch}{1.5}
\begin{tabular}{llll}
  \hline
$\exp (\CA',\nu')$ & $\alpha_H$ & $\exp (\CA'', \nu^*)$ \\
  \hline
  \hline
  $0,\exp(\CA_{\ell-2},m)$ 	& $x_1-x_\ell$ & $\exp(\CA_{\ell-2},m)$ \\
  $1,\exp(\CA_{\ell-2},m)$ 	& $x_2-x_\ell$ & $\exp(\CA_{\ell-2},m)$ \\
  $2,\exp(\CA_{\ell-2},m)$  	& $x_3-x_\ell$ & $\exp(\CA_{\ell-2},m)$ \\
  $3,\exp(\CA_{\ell-2},m)$  	& $\ldots$   & $\ldots$		     \\
  $\vdots$  & $\vdots$  &   $\vdots$ \\
  $\ell-1,\exp(\CA_{\ell-2},m)$  	& $x_1-x_\ell$ & $\exp(\CA_{\ell-2},m)$ \\
  $\ell,\exp(\CA_{\ell-2},m)$   &  $\ldots$   & $\ldots$	    \\
  $\vdots$                                   & $\vdots$   & $\vdots$                \\
  when $m$ is even:        &          &                       \\
  \hline
  $\tfrac{m}{2}(\ell-1)-2,\exp(\CA_{\ell-2},m)$  	& $x_{\ell-2}-x_\ell$ & $\exp(\CA_{\ell-2},m)$ \\
  $\tfrac{m}{2}(\ell-1)-1,\exp(\CA_{\ell-2},m)$  	& $x_{\ell-1}-x_\ell$ & $\exp(\CA_{\ell-2},m)$ \\
  $\tfrac{m}{2}(\ell-1),\exp(\CA_{\ell-2},m)$                &          &                       \\
  \hline
when $m$ is odd:                  &          &                       \\
  \hline
  $\tfrac{m+1}{2}(\ell-1)-2,\exp(\CA_{\ell-2},m)$  	& $x_{\ell-2}-x_\ell$ & $\exp(\CA_{\ell-2},m)$ \\
  $\tfrac{m+1}{2}(\ell-1)-1,\exp(\CA_{\ell-2},m)$  	& $x_{\ell-1}-x_\ell$ & $\exp(\CA_{\ell-2},m)$ \\
  $\tfrac{m+1}{2}(\ell-1),\exp(\CA_{\ell-2},m)$                &          &                       \\
  \hline
\end{tabular}
\medskip  
\caption{Theorem \ref{thm:multibraid}; 
the first $\tfrac{m}{2}$ resp.~$\tfrac{m+1}{2}$ 
rounds of the induction of hyperplanes when $m$ 
is even resp.\ when $m$ is odd and $\ell>3$}
\label{indtable5}
\end{table}

In the subsequent rounds we always get the same restrictions and the 
Euler multiplicities are calculated using Proposition \ref{prop:Euler}(3). 
Let $r$ be the number of rounds of adding the 
hyperplanes $\ker(x_i-x_\ell)$ for $1 \le i \le \ell-1$. 
We consider the next
$\tfrac{m}{2}$ resp.~$\tfrac{m+1}{2}$ 
rounds of the induction of hyperplanes when $m$ 
is even resp.\ when $m$ is odd.
As long as 
$1\leq r\leq\frac{m}{2}$ in the case where $m$ 
is even and $1\leq r\leq\frac{m+1}{2}$ when $m$ is odd,
Proposition \ref{prop:Euler}(3) applies 
and gives the Euler multiplicities,
as shown in Table \ref{indtable5}.

Table \ref{indtable6} 
shows the final $\frac{m}{2}$ rounds in the case where $m$ is even. 
The initial inductively free arrangement $(\CA',\nu')$ here is 
the final arrangement from Table \ref{indtable5}, where 
$\frac{m}{2}$ rounds of adding the hyperplanes 
 $\ker(x_i-x_\ell)$ have already been performed.
Its defining polynomial is 
\[
Q(\CA',\nu') :=
\prod\limits_{1 \le i<j \leq \ell-1}(x_i-x_j)^m
\prod\limits_{1 \le j\leq \ell-1}(x_j-x_\ell)^{m/2}
\]
with set of exponents
\[
\exp(\CA',\nu')=\left\{\tfrac{m}{2}(\ell-1),\exp(\CA_{\ell-2},m)\right\}.
\]
Now we add $\ker(x_1-x_\ell)$ again and get
\[
Q(\CA_{Y^1_a},\nu_{Y^1_a})=(x_1-x_\ell)^{\frac{m}{2}+1}(x_1-x_a)^m(x_a-x_\ell)^{\frac{m}{2}}
\]
and hence $\nu^*(Y^1_a)=m$, by Proposition \ref{prop:Euler}(1).
Moreover, since 
\[
Q(\CA_{Y^1_{b,c}},\nu_{Y^1_{b,c}}) = (x_1-x_\ell)^{\frac{m}{2}+1}(x_b-x_c)^m
\]
and so $\nu^*(Y^1_{b,c})=m$, by Proposition \ref{prop:Euler}(2).
So once again we have $(\CA'',\nu^*) = (\CA_{\ell-2},m)$
with constant multiplicity $m$ again.
This is not any different from all the steps before. 
However, this round's second step, restricting to $\ker(x_2-x_\ell)$, 
leads to a different multiplicity.
For, here we have
\[
Q(\CA_{Y^2_1},\nu_{Y^2_1}) =
(x_1-x_2)^m (x_1-x_\ell)^{\frac{m}{2}+1}(x_2-x_\ell)^{\frac{m}{2}+1}.
\]
It follows from Proposition \ref{prop:Euler}(1)
that $\nu^*(Y^2_1) = m+1$.
One checks that the multiplicity in this case is given by $[m+1,m,\ldots,m]$. 
It follows from the proof of
Lemma \ref{lem:mixed} 
that this restriction $(\CA'', \nu^*)$ is also inductively free, 
because its multiplicity occurs in the 
induction of hyperplanes in the proof of the lemma,
see Table \ref{indtable2}.
Consequently, it has exponents
\[
\left\{\tfrac{m(\ell-1)}{2},\ldots,\tfrac{m(\ell-1)}{2},\tfrac{m(\ell-1)}{2}+1\right\}.
\]
Therefore, at the end of this round the restriction has multiplicity
\[
[\underbrace{m+1,\ldots,m+1}_{(\ell-2)\text{ times}},m,\ldots,m].
\]
Therefore, using our induction hypothesis that Theorem \ref{thm:multibraid} holds
for lower ranks and by Lemma \ref{lem:mixed} applied in case $\ell-1$, we have
\[
\exp(\CA',\nu')=\left\{\tfrac{m(\ell-1)}{2}+1,\ldots,\tfrac{m(\ell-1)}{2}+1\right\}.
\]
In the next round, the restriction's multiplicity 
builds up from $[m+1,\ldots,m+1,m,\ldots,m]$ 
over $[m+2,m+1,\ldots,m+1,m,\ldots,m]$ to 
$[m+2,\ldots,m+2,m,\ldots,m]$ and we can 
argue in the same way as in the preceding round,
using again the induction hypothesis on $\ell$, Lemma \ref{lem:mixed} for the restriction and
the addition part of Theorem \ref{thm:add-del}. 
The same applies for the remaining rounds as the multiplicity of the restriction increases to
\[
\left[\tfrac{3m}{2},\ldots,\tfrac{3m}{2},m,\ldots,m\right]
= \left[m+\tfrac{m}{2},\ldots,m+\tfrac{m}{2},m,\ldots,m\right]
\]
when we restrict to $\ker(x_{\ell-1}-x_\ell)$.
As before we use induction on $\ell$, Lemma \ref{lem:mixed} for the 
restriction and then the addition part of 
Theorem \ref{thm:add-del} again and obtain
the expected exponents $\left\{\frac{m\ell}{2},\ldots,\frac{m\ell}{2}\right\}$,
see Table \ref{indtable6}.

\begin{table}[ht!b]\small
\renewcommand{\arraystretch}{1.5}
\begin{tabular}{llll}
  \hline
$\exp (\CA',\nu')$ & $\alpha_H$ & $\exp (\CA'', \nu^*)$ \\
  \hline
  \hline
$  \frac{m}{2}(\ell-1),\exp(\CA_{\ell-2},m)$ & $x_1-x_\ell$ & 	$\exp(\CA_{\ell-2},m)$ \\
 $ \underbrace{\tfrac{m(\ell-1)}{2},\ldots,\tfrac{m(\ell-1)}{2},\tfrac{m(\ell-1)}{2}+1}_{(\ell-1)\text{ elements}}$ & $x_2-x_\ell$ & 
      	  $\underbrace{\tfrac{m(\ell-1)}{2},\ldots,\tfrac{m(\ell-1)}{2},\tfrac{m(\ell-1)}{2}+1}_{(\ell-2)\text{ elements}}$  \\
$\frac{m(\ell-1)}{2},\ldots,\frac{m(\ell-1)}{2},\frac{m(\ell-1)}{2}+1, \frac{m(\ell-1)}{2}+1$
   &  $\ldots$        & $\ldots$                                                    \\  
 $ \vdots $  	& $\vdots$  & $\vdots$    	             \\
  $\frac{m(\ell-1)}{2}+1,\ldots,\frac{m(\ell-1)}{2}+1$     & $x_1-x_\ell$ &     
 $\frac{m(\ell-1)}{2}+1,\ldots,\frac{m(\ell-1)}{2}+1$   \\
  $\frac{m(\ell-1)}{2}+1,\ldots,\frac{m(\ell-1)}{2}+1, \frac{m(\ell-1)}{2}+2$   & $\ldots$ & $\ldots$ \\
 $ \vdots $  	& $\vdots$  & $\vdots$    	             \\  
 $ \frac{m\ell}{2}-1,\frac{m\ell}{2},\ldots,\frac{m\ell}{2} $      & $x_{\ell-1}-x_\ell$ &     
         $\frac{m\ell}{2},\ldots,\frac{m\ell}{2}$  \\
  $\frac{m\ell}{2},\ldots,\frac{m\ell}{2}$                              \\
\hline
\end{tabular}
\medskip  
\caption{Theorem \ref{thm:multibraid}; induction of hyperplanes for $\ell>3$ and $m$ even}
\label{indtable6}
\end{table}

Table \ref{indtable7} shows the remaining $\frac{m-1}{2}$ rounds in the case where 
$m$ is odd. Comparing to the case where $m$ is even we see that the restriction's
 multiplicity does not change in the course of one round of adding hyperplanes. 
In the $\frac{m+3}{2}$-th round we have
\[
\left[m+1,\ldots,m+1,m,\ldots,m\right]
\]
and in the last round we have
\[
\left[\tfrac{3m-1}{2},\ldots,\tfrac{3m-1}{2},m,\ldots,m\right]
= \left[m+\tfrac{m-1}{2},\ldots,m+\tfrac{m-1}{2},m,\ldots,m\right]
\]
where these multiplicities can be calculated again using Proposition \ref{prop:Euler}(1). 

\[
\left\{\tfrac{(m-1)\ell}{2}+1,\ldots,\tfrac{(m-1)\ell}{2}+\ell-1\right\}.
\]
\begin{table}[ht!b]\small
\renewcommand{\arraystretch}{1.5}
\begin{tabular}{llll}
  \hline
$\exp (\CA',\nu')$ & $\alpha_H$ & $\exp (\CA'', \nu^*)$ \\
  \hline
  \hline
  $\frac{m-1}{2}(\ell-1),\exp(\CA_{\ell-2},m)$ 	& $x_1-x_\ell$ &	$\exp(\CA_{\ell-2},m)$ \\
  \hline
  $\frac{m-1}{2}(\ell-1)+1,\exp(\CA_{\ell-2},m)$    &  $\ldots$       & $\ldots$                           \\
  \hline
 $ \vdots $  	& $\vdots$  & $\vdots$    	             \\    
  \hline
  $\exp(\CA_{\ell-2},m),\frac{m-1}{2}(\ell-1)+\ell-1$    & $x_1-x_\ell$ &   
$\frac{(m-1)(\ell-1)}{2}+2,\frac{(m-1)(\ell-1)}{2}+3,$  \\
& & $\ldots,\frac{(m-1)(\ell-1)}{2}+\ell-1$   \\                                                                              
  \hline
  $\frac{(m-1)(\ell-1)}{2}+2,\frac{(m-1)(\ell-1)}{2}+2,$                            
  & $x_2-x_\ell$ &   $\frac{(m-1)(\ell-1)}{2}+2,\frac{(m-1)(\ell-1)}{2}+3,$   \\
  $\frac{(m-1)(\ell-1)}{2}+3,\ldots,\frac{(m-1)(\ell-1)}{2}+\ell-1$                   
   &         &                  $\ldots,\frac{(m-1)(\ell-1)}{2}+\ell-1$   \\
  \hline
  $\frac{(m-1)(\ell-1)}{2}+2,\frac{(m-1)(\ell-1)}{2}+3,$                            &  $\ldots$       & $\ldots$  \\
  $\frac{(m-1)(\ell-1)}{2}+3,\ldots,\frac{(m-1)(\ell-1)}{2}+\ell-1$     \\
  \hline
 $ \vdots $  	& $\vdots$  & $\vdots$    	             \\      
  \hline
  $\frac{(m-1)(\ell-1)}{2}+2,\ldots,\frac{(m-1)(\ell-1)}{2}+\ell$           & $x_1-x_\ell$ &         
         $\frac{(m-1)(\ell-1)}{2}+3,\ldots, \frac{(m-1)(\ell-1)}{2}+\ell$ \\
\hline
  $\frac{(m-1)(\ell-1)}{2}+3,\frac{(m-1)(\ell-1)}{2}+3,$                              & $x_2-x_\ell$ &                
     $\frac{(m-1)(\ell-1)}{2}+3,\ldots, \frac{(m-1)(\ell-1)}{2}+\ell$ \\
 $ \ldots,\frac{(m-1)(\ell-1)}{2}+\ell$                         & &  \\                                        
  \hline
  $\frac{(m-1)(\ell-1)}{2}+3,\frac{(m-1)(\ell-1)}{2}+4,$       &  $\ldots$       & $\ldots$      \\
  $\frac{(m-1)(\ell-1)}{2}+4,\ldots,\frac{(m-1)(\ell-1)}{2}+\ell$           \\ 
  \hline
 $ \vdots $  	& $\vdots$  & $\vdots$    	             \\      
  \hline
  $\frac{(m-1)\ell}{2},\frac{(m-1)\ell}{2}+1, \ldots,\frac{(m-1)\ell}{2}+\ell-2$ 
 & $x_1-x_\ell$ &    $\frac{(m-1)\ell}{2}+1,\ldots,\frac{(m-1)\ell}{2}+\ell-2$  \\
  \hline
  $\frac{(m-1)\ell}{2}+1,\frac{(m-1)\ell}{2}+1, \ldots,\frac{(m-1)\ell}{2}+\ell-2$    &  $\ldots$       & $\ldots$   \\
  \hline
 $ \vdots $  	& $\vdots$  & $\vdots$    	             \\        
  \hline
  $\frac{(m-1)\ell}{2}+1,\ldots,\frac{(m-1)\ell}{2}+\ell-2, \frac{(m-1)\ell}{2}+\ell-2$    
   & $x_{\ell-1}-x_\ell$ &  $\frac{(m-1)\ell}{2}+1,\ldots,\frac{(m-1)\ell}{2}+\ell-2$  \\
  \hline
  $\frac{(m-1)\ell}{2}+1,\ldots,\frac{(m-1)\ell}{2}+\ell-1$   &    &                          \\      
  \hline  
\end{tabular}
\medskip  
\caption{Theorem \ref{thm:multibraid}; induction of hyperplanes for $\ell>3$ and $m$ odd}
\label{indtable7}
\end{table}

Consequently, the restriction's exponents do not change during any of these rounds either. 
So there is only one element in the set $\exp(\CA',\nu')$ that increases by $1$ in 
every step.

As before, arguing by induction on $\ell$, 
employing Lemma \ref{lem:mixed} and the addition part of 
Theorem \ref{thm:add-del} in each round, 
we obtain the expected exponents
This completes the proof of Theorem \ref{thm:multibraid}.

Corollary \ref{cor:mixedmulti} follows from 
Theorem \ref{thm:multibraid} and 
Lemma \ref{lem:mixed}.


\bigskip
{\bf Acknowledgments}:
We would like to thank the anonymous referee for making 
a number of useful comments and for bringing to our attention that our
result can in principle be recovered from \cite{abenudanumata}.


\bigskip

\bibliographystyle{amsalpha}

\newcommand{\etalchar}[1]{$^{#1}$}
\providecommand{\bysame}{\leavevmode\hbox to3em{\hrulefill}\thinspace}
\providecommand{\MR}{\relax\ifhmode\unskip\space\fi MR }
\providecommand{\MRhref}[2]{%
  \href{http://www.ams.org/mathscinet-getitem?mr=#1}{#2} }
\providecommand{\href}[2]{#2}


\end{document}